\documentclass[12pt]{amsart}

\usepackage{mathtools}

\newtheorem{theorem}{Theorem}
\newtheorem{lemma}{Lemma}
\newtheorem{proposition}{Proposition}
\newtheorem{corollary}{Corollary}
\theoremstyle{definition}
\newtheorem{definition}{Definition}
\theoremstyle{remark}

\newtheorem*{acknowledgements}{Acknowledgements}

\def\F{\mathcal{F}}

\begin{document}


\title{Spectral Properties of Small Hadamard Matrices}
\author{Dorin Ervin Dutkay}
\author{John Haussermann}
\author{Eric Weber}

\address{[Dorin Ervin Dutkay] University of Central Florida\\
	Department of Mathematics\\
	4000 Central Florida Blvd.\\
	P.O. Box 161364\\
	Orlando, FL 32816-1364\\
U.S.A.} \email{Dorin.Dutkay@ucf.edu}

\address{[John Haussermann] University of Central Florida\\
	Department of Mathematics\\
	4000 Central Florida Blvd.\\
	P.O. Box 161364\\
	Orlando, FL 32816-1364\\
U.S.A.} \email{jhaussermann@knights.ucf.edu}
\address{[Eric Weber] Iowa State University\\
Department of Mathematics\\ 396 Carver Hall\\ Ames, IA 50011, U.S.A.}\email{esweber@iastate.edu}

\subjclass[2010]{05B20,65T50,11L05}
\keywords{Hadamard matrix, Discrete Fourier Transform, Gauss sums, eigenvalues}
\date{\today}
\begin{abstract}
We prove that if $A$ and $B$ are Hadamard matrices which are both of size $4 \times 4$ or $5 \times 5$ and in dephased form, then $tr(A) = tr(B)$ implies that $A$ and $B$ have the same eigenvalues, including multiplicity.  We calculate explicitly the spectrum for these matrices.  We also extend these results to larger Hadamard matrices which are permutations of the Fourier matrix and calculate their spectral multiplicities.
\end{abstract}
\maketitle



\section{Introduction}

The matrix of the discrete Fourier transform for the finite group $\mathbb{Z}_n$, where $n\geq2$ is an integer, is 
$$\F_n:=\frac{1}{\sqrt{n}}\left(e^{\frac{2\pi i jk}{n}}\right)_{j,k}.$$
The problem of finding the eigenvalues and their multiplicity for the discrete Fourier transform has a long history (see, e.g. \cite{AuTo79}) and has connections to many areas such as harmonic analysis, number theory, and numerical analysis. For example, one can see immediately that the trace of the matrix of the discrete Fourier transform is the Gauss sum 
$$\sum_{j=0}^{n-1}e^{\frac{2\pi ij^2}{n}}.$$
It took Gauss several years to give a complete formula for these sums. And it can be used to give a short proof for the famous quadratic reciprocity law, which can be expressed in the following formula:
$$\left(\begin{array}{c}p\\-\\q\end{array}\right)\left(\begin{array}{c}q\\-\\p\end{array}\right)=\frac{\mbox{Tr}(\F_{pq})}{\mbox{Tr}(\F_p)\mbox{Tr}(\F_q)}.$$

The eigenvalues and eigenvectors for the Fourier matrix were computed (without the use of Gauss sums!) by McClelland and Parks in the paper \cite{McPa72} which appeared in IEEE Transactions on Audio and Electroacoustics. A simpler computation of the multiplicities of the eigenvalues is possible with the use of Gauss sums (see \cite{AuTo79}).  Thus, the spectrum of the Fourier matrix unveils some deep results in number theory. In this paper we calculate the spectrum of several classes of Hadamard matrices.

\begin{definition}
An $n\times n$ matrix $H$ is called a {\it Hadamard matrix} if it is unitary, i.e., $HH^*=H^*H=I_n$, and all the entries have complex modulus $\frac{1}{\sqrt{n}}$.  In some contexts, such an $H$ is called a complex Hadamard matrix, but we will not make a distinction between real and complex $H$.  Also, in some contexts, a Hadamard matrix is orthogonal but not unitary.
\end{definition}

Hadamard matrices appear in a number of contexts. See \cite{TaZy06} for a history of their development and many examples of their uses, including computing  and quantum physics. Recent applications to quantum permutation groups are discussed in \cite{Ban12}. For example, a correspondence is shown between symmetries of a Hadamard matrix and quantum permutation groups.  We are motivated by the appearance of Hadamard matrices in the context of the Fuglede conjecture \cite{Fug74,DH12} and Fourier analysis on fractals \cite{JP98}.

Note that by virtue of $H$ being unitary, the spectrum is a subset of the unit circle.   Every Hadamard matrix $H$ can be factored as $H = D_{1} H_{0} D_{2}$ where $D_{1}$ and $D_{2}$ are diagonal matrices with diagonal entries of modulus $1$, and $H_{0}$ is again a Hadamard matrix, but in \emph{dephased form}, that is, all the entries in the first row and first column are $\frac{1}{\sqrt{n}}$.  Henceforth, we will assume that a Hadamard matrix is in dephased form.

For convenience, we will index the rows and columns of an $n \times n$ matrix by $\{0, 1, \dots, n-1\}$.  For a Hadamard matrix $H$ in dephased form, the $0$-th row and column have entries which are all $\frac{1}{\sqrt{n}}$; the remaining entries of the matrix, which form an $n-1 \times n-1$ principal submatrix, is called the \emph{core} of $H$.

The matrix of the usual discrete Fourier transform is a Hadamard matrix, and it is called {\it the standard Hadamard matrix} of dimension $n$.  Any permutation of the Fourier matrix gives also a Hadamard matrix (though not necessarily dephased). Two Hadamard matrices are called equivalent if they can be obtained from each other by permutations of row and columns and multiplication by diagonal matrices with diagonal entries of modulus 1. Hadamard matrices of dimensions up to and including 5 have been classified up to equivalence, but even for dimension 6, no such classification exists (see \cite{TaZy06}).

\begin{definition}
We say that two matrices $A$ and $B$ are permutationally equivalent if there are permutation matrices $P$ and $Q$ such that $B = Q^{T} A P$.  We say the  two matrices are spectrally equivalent if their spectra are equal including multiplicity, or equivalently $A = U^{*} B U$ for some unitary $U$.
\end{definition}
Note that in general when $A$ and $B$ are permutationally equivalent, they need not be spectrally equivalent.  Note also that if $P = Q$, then the diagonal entries of $B$ are the same as the diagonal entries of $A$, but their positions have potentially been permuted.

We will describe the spectrum of the Hadamard matrices of dimensions 4 and 5 and of symmetric permutations of the Fourier matrix. This will be done by showing the following metatheorem:  If $A$ and $B$ are two Hadamard matrices of a common class, then $tr(A)=tr(B)$ implies that $A$ and $B$ are spectrally equivalent.

We denote the entries of a matrix $M$ by $M[j,k]$.  For consistent notation of permutation matrices, we define for $\sigma$ a permutation of $\{0, 1, \dots, n-1\}$ the permutation matrix $P_{\sigma}$ by $P_{\sigma}[j,k] = 1$ if $j = \sigma(k)$ and $0$ otherwise.  Thus, for a matrix $M$,
\begin{align*}
(M P_{\sigma}) [j,k] &= \sum_{l = 0}^{n} M[j,l] P_{\sigma}[l,k] \\
&= M[j, \sigma(k) ].
\end{align*}
Likewise,
\begin{align*}
(P_{\sigma}^{T} M) [j,k] &= \sum_{l = 0}^{n} P_{\sigma}^{T}[j,l] M[l,k] \\
&= \sum_{l = 0}^{n} P_{\sigma}[l,j] M[l,k] \\
&= M[\sigma(j),k].
\end{align*}
Since $P_{\sigma}^{T} P_{\sigma} = I$, $P_{\sigma}^{T} = P_{\sigma^{-1}}$.

\section{$4 \times 4$ Hadamard matrices}

The basic $4 \times 4$ Hadamard matrix is:
\[ H(\rho) = \dfrac{1}{2}
\begin{pmatrix*}[r] 
1 & 1 & 1 & 1 \\ 1 & \rho & -1 & -\rho \\ 1 & -1 & 1 & -1 \\ 1 & -\rho & -1 & \rho \end{pmatrix*}
\]
where $\rho$ is any complex number of modulus $1$.  Any $4 \times 4$ Hadamard matrix in dephased form is permutationally equivalent to $H(\rho)$, for some $\rho$ of absolute value 1 (\cite{TaZy06}).  An easy computation yields that the eigenvalues of this matrix are $\{ 1, 1, -1, \rho\}$, with the corresponding eigenvectors:
\[  
\begin{pmatrix*}[r] 3 \\ 1 \\ 1 \\ 1 \end{pmatrix*}, 
\quad 
\begin{pmatrix*}[r] 0 \\ -1 \\ 2 \\ -1\end{pmatrix*}, 
\quad 
\begin{pmatrix*}[r] 1 \\ -1 \\ -1 \\ -1 \end{pmatrix*}, 
\quad 
\begin{pmatrix*}[r] 0 \\ 1 \\ 0 \\ -1 \end{pmatrix*}.  
\]
Note that the eigenvectors are independent of $\rho$, so the matrices $H(\rho)$ form a commutative set.

Our main theorem for $4 \times 4$ Hadamard matrices is the following.
\begin{theorem} \label{T:main4}
If $A$ and $B$ are both $4 \times 4$ Hadamard matrices in dephased form, and $tr(A) = tr(B)$, then $A$ and $B$ are spectrally equivalent.
\end{theorem}

We proceed with the proof of this by considering several cases.  We first consider when $A$ and $B$ are both real, i.e. have all real entries.  We then consider the case when $A$ and $B$ are both symmetric; and finally when $A$ and $B$ are both non-symmetric.  Our first case is justified by the following lemma.

\begin{lemma} \label{L:trace4}
If $A$ is a $4 \times 4$ Hadamard matrix in dephased form, and if $A$ has a non-real entry, then the trace of $A$ is non-real.
\end{lemma}

\begin{proof}
Since $A$ is permutationally equivalent to $H(\rho)$ for some $\rho$, if $A$ has any non-real entries, then they must be $\rho$ and $-\rho$, and both must appear in pairs.  Moreover, since any row, likewise column, which contains $\rho$, must also contain $-\rho$, the non-real entries of $A$ form a $2 \times 2$ submatrix.  It follows that there are one or two non-real entries on the diagonal of $A$.  If there is only one non-real entry on the diagonal of $A$, then the trace is non-real.  If there are two non-real entries on the diagonal of $A$, then they must be both $\rho$ or both $-\rho$, and thus the trace of $A$ is non-real.
\end{proof}

\begin{proposition} \label{P:real4}
If $A$ and $B$ are $4 \times 4$ Hadamard matrices in dephased form and have common trace which is real, then they are spectrally equivalent.
\end{proposition}

\begin{proof}
By Lemma \ref{L:trace4}, both $A$ and $B$ are real.  The core of $A$ and $B$ must have a $1$ in each row and column, and two $-1$'s in each row and column.  Thus, the core of $A$ and $B$ must be one of the following six possibilities (without the scaling factor of $\frac{1}{2}$):
\begin{align*}
C_{1} &= \begin{pmatrix*}[r] 1 & -1 & -1 \\ -1 & 1 & -1 \\ -1 & -1 & 1 \end{pmatrix*}
\quad
C_{2} = \begin{pmatrix*}[r] -1 & 1 & -1 \\ -1 & -1 & 1 \\ 1 & -1 & -1 \end{pmatrix*}
\quad
C_{3} = \begin{pmatrix*}[r]  -1 & -1 & 1 \\ 1 & -1 & -1 \\ -1 & 1 & -1 \end{pmatrix*} \\
C_{4} &= \begin{pmatrix*}[r] -1 & 1 & -1 \\ 1 & -1 & -1 \\ -1 & -1 & 1 \end{pmatrix*}
\quad
C_{5} = \begin{pmatrix*}[r] -1 & -1 & 1 \\ -1 & 1 & -1 \\ 1 & -1 & -1 \end{pmatrix*}
\quad
C_{6} = \begin{pmatrix*}[r] 1 & -1 & -1 \\ -1 & -1 & 1 \\ -1 & 1 & -1 \end{pmatrix*}.
\end{align*}
Note that if the core of $A$ is $C_{1}$, then $tr(A) = 2$; if the core of $A$ is $C_{2}$ or $C_{3}$, then $tr(A) = \frac{-1}{2}$, and if the core of $A$ is any of the remaining three, then $tr(A) = 0$.  Thus, we need to show that $C_{2}$ and $C_{3}$ are spectrally equivalent cores, and likewise $C_{4}$, $C_{5}$, and $C_{6}$.

Note that $C_{3}^{T} = C_{2}$, so they are spectrally equivalent, i.e. if $C_{3}$ is the core of $A$ and $C_{2}$ is the core of $B$, then $A$ and $B$ are spectrally equivalent.

Now, suppose the core of $A$ is $C_{4}$ and the core of $B$ is $C_{5}$.  Let $\sigma = (123)$; then $P_{\sigma}^{T} A P_{\sigma} = B$.  Likewise, if the core of $B$ is $C_{6}$, and if $\tau = (13)$, then $P_{\tau}^{T} A P_{\tau} = B$.  Thus, $C_{4}$, $C_{5}$, and $C_{6}$ are spectrally equivalent cores.
\end{proof}

For the remainder of this section, we will assume that the Hadamard matrices have non-real entries; the non-real entries are $\rho$ and $-\rho$.  By the proof of Lemma \ref{L:trace4}, it is impossible for both $\rho$ and $-\rho$ to be on the diagonal; likewise, it is impossible for the non-real entries all to be off-diagonal.

\begin{proposition} \label{P:sym4}
Let $H$ be a symmetric $4 \times 4$ Hadamard matrix in dephased form with $\rho$ on the diagonal.  Then the eigenvalues of $H$ are $\{1, 1, -1, \rho\}$.
\end{proposition}

\begin{proof}
Since $H$ is symmetric, the $1$ that appears in the core is on the diagonal. Since all the $-1$'s in the matrix are on the same row or column as this $1$, this completely determines the matrix, except for switching $\rho$ with $-\rho$.   Suppose that the $1$ is in the $(1,1)$ position; if $\sigma = (12)$, then the matrix $P_{\sigma}^{T} H P_{\sigma}$ has the $1$ in the $(2,2)$ position.  Moreover, the other two entries on the diagonal are $\rho$; thus $P_{\sigma}^{T} H P_{\sigma} = H(\rho)$.  It follows that $H$ is spectrally equivalent to $H(\rho)$, and as noted in the introduction, the eigenvalues of $H(\rho)$ are $\{1, 1, -1, \rho\}$.

An analogous argument applies if $H$ has the $1$ in the $(3,3)$ position.
\end{proof}

\begin{lemma} \label{L:nonsym4}
Suppose $A$ and $B$ are non-symmetric $4 \times 4$ Hadamard matrices. Then $B$ can be obtained from $A$ by the following operations:  conjugation by a permutation matrix, transposition, or interchanging $\rho$ and $-\rho$.
\end{lemma}

\begin{proof}
Consider the core of the matrices $A$ and $B$; since they are non-symmetric, the cores have the $1$ off the diagonal.  We may, if necessary, transpose $A$ and/or $B$ to place the $1$ above the diagonal for both.  Also, if necessary, we may interchange $\rho$ and $-\rho$ in $A$ and/or $B$ so that the $-\rho$'s are off the diagonal.  Thus, the core of $A$ and $B$ must be one of the following three possibilities:
\[ D_{1} = \begin{pmatrix*}[r] -1 & 1 & -1 \\ \rho & -1 & -\rho \\ -\rho & -1 & \rho \end{pmatrix*}
\quad
D_{2} = \begin{pmatrix*}[r] -1 & -1 & 1 \\ -\rho & \rho & -1 \\ \rho & -\rho & -1 \end{pmatrix*}
\quad
D_{3} = \begin{pmatrix*}[r] \rho & -\rho & -1 \\ -1 & -1 & 1 \\ -\rho & \rho & -1 \end{pmatrix*}.
\]
If $\sigma = (23)$, then $P_{\sigma}^{T} D_{1} P_{\sigma} = D_{2}$; if $\sigma = (12)$, then $P_{\sigma}^{T} D_{2} P_{\sigma} = D_{3}$.  Thus, the core of $A$ can be interchanged to the core of $B$ by the three operations as stated.
\end{proof}

We are now ready to prove Theorem \ref{T:main4}.
\begin{proof}
If the common trace of $A$ and $B$ is real, then by Proposition \ref{P:real4}, $A$ and $B$ are spectrally equivalent.  If the trace is non-real, then we must demonstrate that either $A$ and $B$ are both symmetric, or both non-symmetric.   By the proof of Lemma \ref{L:trace4}, if $A$ is symmetric, then the trace must be $1 \pm \rho$.  Likewise, if $B$ is non-symmetric, then by the proof of Lemma \ref{L:nonsym4}, the trace must be $\frac{-1 \pm \rho}{2}$.  For $\rho$ non-real, these traces cannot be equal.

If they are both symmetric, with common trace, then they have the same non-real diagonal entries, and so Proposition \ref{P:sym4} implies that they are spectrally equivalent.  Likewise, if they are both non-symmetric, then by Lemma \ref{L:nonsym4}, $B$ can be obtained from $A$ by the three operations, but since they have common trace, we do not need to interchange $\rho$ and $-\rho$.   Hence, they are spectrally equivalent.
\end{proof}

We have already computed explicitly the eigenvalues of a symmetric, non-real $4 \times 4$ Hadamard matrix.  We consider the other possibilities now.

\begin{corollary}
If $H$ is a real $4 \times 4$ Hadamard matrix in dephased form, then the possible sets of eigenvalues are: $\{ 1 , 1, 1, -1\}$, $\{1,1, -1, -1\}$ or $\{1, -1, e^{2 \pi i/3}, e^{4 \pi i /3} \}$.
\end{corollary}

\begin{proof}
By Proposition \ref{P:real4}, we need to compute the eigenvalues of Hadamard matrices which have cores $C_{1}$, $C_{2}$, or $C_{4}$.  Computations yield that the core $C_{1}$ has eigenvalues of $\{ 1 , 1, 1, -1\}$, the core $C_{2}$ (and $C_{3}$) has eigenvalues of $\{1, -1, e^{2 \pi i/3}, e^{4 \pi i /3} \}$, and the core $C_{4}$ (and $C_{5}$, $C_{6}$) has eigenvalues $\{1,1, -1, -1\}$.
\end{proof}

\begin{corollary}
If $H$ is a non-symmetric $4 \times 4$ Hadamard matrix in dephased form with non-real entry $\rho$,  then the trace of $H$ is either $\frac{-1+\rho}{2}$ or $\frac{-1-\rho}{2}$ and the eigenvalues of $H$ are, respectively,
\[ \{ 1, -1, \dfrac{ -(1 + \rho) \pm \sqrt{ 1 - 14 \rho + \rho^2} }{4} \},
\qquad
\text{ or }
\qquad
\{ 1, -1, \dfrac{ -(1 - \rho) \pm \sqrt{ 1 + 14 \rho + \rho^2} }{4} \}. \]
\end{corollary}

\begin{proof}
If $H$ is not symmetric and has trace $\frac{-1 + \rho}{2}$, then by Theorem \ref{T:main4} it has the same eigenvalues as 
\[ H_{1}(\rho) = \dfrac{1}{2}
\begin{pmatrix*}[r]
1 & 1 & 1 & 1  \\ 1 & -1 & 1 & -1 \\ 1 & \rho & -1 & -\rho \\ 1 & -\rho & -1 & \rho 
\end{pmatrix*}.
\]
Likewise, if $tr(H) = \frac{-1 - \rho}{2}$, then it has the same eigenvalues as $H_{1}(-\rho)$.
A computation produces those eigenvalues.
\end{proof}

To quickly determine the eigenvalues of a $4 \times 4$ dephased Hadamard matrix from among the above possibilities, we can use the fact that the trace of a matrix is the sum of its eigenvalues. In the case that such a matrix is symmetric, we obtain that the sign of the last eigenvalue matches the sign of the entries on the diagonal that are not equal to $1$. In the case of an asymmetric matrix, there are two $-1$'s and a $1$ on the diagonal, and $\rho$ or $-\rho$. For the sake of notation, say this (other) symbol on the diagonal is $x$. Then
\[
\frac{-1+x}{2} = 1 + -1 + \dfrac{ -(1 \pm \rho) + \sqrt{ 1 \mp 14 \rho + \rho^2} }{4}  + \dfrac{ -(1 \pm \rho) - \sqrt{ 1 \mp 14 \rho + \rho^2} }{4} .
\]
Simplifying, we obtain
\[
\frac{-1+x}{2} = \frac{-1\pm \rho}{2} .
\]
Therefore, the sign which is not under the square root of the $\pm$ in the expression for the eigenvalues of the asymmetric matrix is the same as the sign of the element on the diagonal which is not $1$ or $-1$. In the case that $x = \pm 1$, this still holds: if $x = -1$, then all the elements on the diagonal except the first are negative, so there is no ambiguity, and if $x = 1$ then the matrix is symmetric.

We now investigate the expression providing the complicated eigenvalues in the asymmetric case.

\begin{lemma}
The functions $k_{\pm}(\rho) = \dfrac{ -(1 + \rho) \pm \sqrt{ 1 - 14 \rho + \rho^2} }{4}$ map the unit circle to the unit circle, and are surjective. 
\end{lemma}
\begin{proof}
Hadamard matrices have eigenvalues on the unit circle, so if $\rho$ is on the unit circle, so is $k(\rho)$.

We solve the expression 
\[
k = \dfrac{ -(1 + \rho) \pm \sqrt{ 1 - 14 \rho + \rho^2} }{4}
\]
for $\rho$ and obtain
\[
\rho = -k \left( \dfrac{k+\frac12}{1+\frac12 k} \right) .
\]
The right hand side is $-k$ times a Mobius transform which fixes the unit circle. Therefore, for any $k$ on the unit circle, there is a corresponding $\rho$ on the unit circle. So $k$ is onto.
\end{proof}

\section{$5 \times 5$ Hadamard matrices}

Any $5 \times 5$ Hadamard matrix in dephased form is permutationally equivalent to the Fourier matrix on $\mathbb{Z}_{5}$:
\[
\mathcal{F}_{5} = 
\frac{1}{\sqrt{5}}
\begin{pmatrix*}
1 & 1 & 1 & 1 & 1 \\
1 & \omega  & \omega^2 & \omega^3 & \omega^4  \\
1 & \omega^2  & \omega^4  & \omega & \omega^3  \\
1 &  \omega^3 & \omega & \omega^4 & \omega^2  \\
1 &  \omega^4 & \omega^3 & \omega^2 & \omega 
\end{pmatrix*}
\]
where $\omega = e^{-2 \pi i / 5}$ (see Haagerup \cite{Haag97}, \cite{Haag12}).  Any such matrix must be a latin square in the symbols $1,\omega, \dots, \omega^4$, as a consequence of the following statement.

\begin{proposition} \label{P:linind}
If $\omega$ is a fifth root of unity, and $n_{1}, n_{2}, n_{3}, n_{4}$ are distinct integers modulo 5, then $\{ \omega^{n_{1}}, \omega^{n_{2}}, \omega^{n_{3}}, \omega^{n_{4}} \}$ is linearly independent over $\mathbb{Q}$.
\end{proposition}

\begin{proof}
This is a well-known result from elementary number theory, but we include the key points of the proof here.  The field $\mathbb{Q}[\omega]$ is a field extension of $\mathbb{Q}$ of dimension 4, with basis $\{1, \omega, \omega^2, \omega^3 \}$.  It follows that any 4 distinct powers (modulo 5) of $\omega$ is linearly independent over $\mathbb{Q}$.
\end{proof}

Our main result is analogous to the $4 \times 4$ case.

\begin{theorem}  \label{T:5x5H}
If $A$ and $B$ are two $5 \times 5$ Hadamard matrices in dephased form, and $tr(A) = tr(B)$, then $A$ and $B$ are spectrally equivalent.
\end{theorem}

We will proceed by considering symmetric and non-symmetric matrices separately.  This will not cause a problem, because we will show in due course that if $tr(A) = tr(B)$, then either both are symmetric or both are non-symmetric.  In the non-symmetric case, we will focus on the core part of the matrices, which will involve case analysis.  We will consider incomplete core parts, which will have non-unique completions--if the determined positions in the incomplete core part do not uniquely determine the remaining positions, we will have a \emph{pivot position}, an entry in the matrix that has several possibilities.  We will complete the core part to obtain a latin square, but there will be instances when the resulting matrix will not be a Hadamard matrix.  For example, the ``core'' of the following dephased matrix
\[ \frac{1}{\sqrt{5}}
\begin{pmatrix*}
1 & 1 & 1 & 1 & 1 \\
1 & \omega  & \omega^2 & \omega^3 & \omega^4  \\
1 & \omega^2  & \omega  & \omega^4 & \omega^3  \\
1 &  \omega^3 & \omega^4 & \omega^2 & \omega  \\
1 &  \omega^4 & \omega^3 & \omega & \omega^2 
\end{pmatrix*}
\]
is a latin square, but the matrix is not Hadamard because it is not unitary (consider the inner product of the last two rows).  

Since the core of a Hadamard matrix is a latin square, for any two rows $j$ and $j'$, there exists a permutation $\tau$ of $\{1,2,3,4\}$ such that the core of $H$ satisfies the equation $H[j,l] =H[j', \tau(l)]$.  

\begin{definition} \label{D:pairup}
We say that two rows $j$, $j'$ of $H$ \emph{pair up} if the permutation $\tau$ of $\{1,2,3,4\}$ such that the core of $H$ satisfies the equation 
\[ H[j,l] = H[j',\tau(l)] \]
has order 2.
\end{definition}

We have the following structural restriction for the core part of a $5 \times 5$ Hadamard matrix.

\begin{proposition} \label{P:pairup}
If $H$ is a $5 \times 5$ Hadamard matrix and rows $j$ and $j'$ of $H$ pair up, then we must have for every $k=1,\dots,4$,
\[ \tilde{H}[j,k] + \tilde{H}[j',k]  \equiv 0 \mod 5.\]
where $\tilde{H}[j,k]\in\{0,1,2,3,4\}$ denotes the power of $\omega$ that appears in the entry $H[j,k]$ of the matrix $H$.
\end{proposition}

\begin{proof}
Without loss of generality, suppose that rows 1 and 2 of  $H$ pair up.  If necessary, we may permute the columns of $H$ so that in Definition \ref{D:pairup}, $\tau = (12)(34)$.
Thus, we have the following core of $H$:
\[ 
\sqrt{5} H_c = 
\begin{pmatrix*}[r]
a & b & c & d \\
b & a & d & c \\
* & * & * & * \\
* & * & * & *
\end{pmatrix*}
\]
If necessary, permute the columns of $H$ again so that $a = \omega$, and denote $b = \omega^{n_2}$, $c = \omega^{n_3}$, and $d = \omega^{n_4}$.  We have
\begin{equation} \label{E:ortho}
1 + \omega^{1-n_2} + \omega^{n_2 -1} + \omega^{n_3 - n_4} + \omega^{n_4 - n_3} = 0, 
\end{equation}
so by Proposition \ref{P:linind}, none of the exponents in the above equation can be repeated.  If $n_{2} = 2$, then $|n_{3} - n_{4}| = 1 = n_{2} - 1$, which would violate the uniqueness of the exponents.  Likewise, if $n_{2} = 3$, then $|n_{3} - n_{4}| = 2 = n_2 - 1$, again a violation.  Thus, we must have $n_{2} = 4$ and $\{n_{3}, n_{4} \} = \{2,3\}$.
\end{proof}

\begin{corollary} \label{C:pairup}
Suppose $K$ is a $5 \times 5$ matrix whose entries are fifth roots of unity and whose column $0$ and row $0$ contain only $1$'s.  If one row of $K$ pairs up with two other rows, then $\frac{1}{\sqrt{5}}K$ is not Hadamard.
\end{corollary}

\begin{proof}
If the matrix $\sqrt{5} K$ is not a latin square, it is not a Hadamard matrix.  Thus, suppose $\sqrt{5} K$ is a latin square.  Permute the rows of $K$ so that row 1 and row 2 of $K$ pair up, as well as row 1 and row 3.  By Proposition \ref{P:pairup}, we have $\tilde{K}[2,1] = -\tilde{K}[1,1] = \tilde{K}[3,1]$ modulo $5$, thus  $\sqrt{5} K$ is not a latin square, a contradiction.
\end{proof}

\begin{lemma} \label{L:diag}
If $A$ and $B$ are $5 \times 5$ Hadamard matrices in dephased form, and $tr(A) = tr(B)$, then $diag(A) = diag(B)$, counting multiplicities.
\end{lemma}

\begin{proof}
We write $tr(A) = \sum_{j=0}^{4} m_{j} \omega^{j} = 1 + \sum_{j=1} m_{j} \omega^{j}$ and $tr(B) = \sum_{j=0}^{4} n_{j} \omega^{j} = 1 + \sum_{j=1}^{4} n_{j} \omega^j$, where $m_{j}$ and $n_{j}$ are the multiplicities of $\omega^{j}$ on the diagonal of $A$ and $B$, respectively.  We have $0 = tr(A) - tr(B) = \sum_{j=1}^{4} (m_{j} - n_{j}) \omega^{j}$; by Proposition \ref{P:linind}, we must have $m_{j} = n_{j}$ for all $j$.
\end{proof}

We are now in position to prove Theorem \ref{T:5x5H}.

\begin{proof}  Our assumption is that $tr(A) = tr(B)$.  By Lemma \ref{L:diag}, the trace of $A$ uniquely determines the diagonal entries, thus we will prove that any two matrices with the same diagonal are spectrally equivalent.  We will consider here non-symmetric $A$; in Section  \ref{S:nxn} we will prove that if $A$ is a symmetric $5 \times 5$ Hadamard matrix then it is spectrally equivalent to either $\mathcal{F}_{5}$ or $\mathcal{F}_{5} P_{(\cdot 2)}$ in Theorem   \ref{T:fnperm}.

Therefore, we proceed now under the assumption that $A$ is a nonsymmetric matrix; we proceed by considering separate cases of the diagonal of the core part of $A$ having 4 of a kind, 3 of a kind, 2 pair, 1 pair and 2 singletons, and finally 4 singletons.  We proceed by considering first the latin square property of the core of $A$, and determine all possible latin squares with a given diagonal and, when possible, first row.  For a matrix with given diagonal and first row, when there is an entry which is not uniquely determined, we will \emph{pivot} on that position, i.e. we will consider all possible entries for that position.  We aim to prove that when there are several ways of completing the core to obtain a latin square which \emph{potentially} result in Hadamard matrices, then they are spectrally equivalent.  Note that we are not checking whether a latin square actually is the core of a Hadamard matrix; we are merely showing that if the cores do yield Hadamard matrices, then they are spectrally equivalent.

\noindent \textbf{4 of a kind:}  Consider the incomplete matrix:
\[ A_{4}(*) =
\begin{pmatrix*}[r]
a  & b & c & d \\
*  & a &  &  \\
  &  & a &  \\
  &  &  & a 
\end{pmatrix*}.
\]
Pivoting on the $*$ position, we have the following posibilities:
\[
[A_{4}(b)](*) =
\begin{pmatrix*}[r]
a  & b & c & d \\
b  & a & d & c \\
*  &  & a & b \\
  &  & b & a 
\end{pmatrix*}
\qquad
A_{4}(c) =
\begin{pmatrix*}[r]
a  & b & c & d \\
c  & a & d & b \\
b  & d & a & c \\
d  & c & b & a 
\end{pmatrix*}
\qquad
A_{4}(d) = 
\begin{pmatrix*}[r]
a  & b & c & d \\
d  & a & b & c \\
c  & d & a & b \\
b  & c & d & a 
\end{pmatrix*}.
\]
Note that $[A_{4}(b)](c)$ results in a symmetric matrix, which we are not allowing here, but which is impossible since it also violates Corollary \ref{C:pairup}.  Thus, we have three possible completions: $A_{4}(c)$, $A_{4}(d)$, and 
\[
[A_{4}(b)](d) =
\begin{pmatrix*}[r]
a  & b & c & d \\
b  & a & d & c \\
d  & c & a & b \\
c  & d & b & a 
\end{pmatrix*}.
\]

Now, let $\sigma$ be the permutation $(3 4)$. 
\[ P_{\sigma}^{T} A_{4}(d) P_{\sigma} =
\begin{pmatrix*}[r]
a  & b & d & c \\
d  & a & c & b\\
b  & c & a & d \\
c  & d & b & a 
\end{pmatrix*}.
\]
Note that in $A_{4}(c)$, $d \equiv -a \mod 5$, and in $P^{T} A_{4}(d) P$, $c \equiv -a \mod 5$ by Proposition \ref{P:pairup}.  Therefore, both matrices have anti-diagonal entries that are equal to each other.  Therefore, it follows that either $P_{\sigma}^{T} A_{4}(d) P_{\sigma} = A_{3}(c)$ or $P_{\sigma}^{T} A_{4}(d) P_{\sigma} = A_{3}(c)^{T}$, which, in either case, implies that $A_{4}(c)$ and $A_{4}(d)$ are spectrally equivalent.

Now let $\rho$ be the permutation $(1432)$; we have
\[ P_{\rho}^{T} [A_{4}(b)](d) P_{\rho} = 
\begin{pmatrix*}[r]
a  & c & d & b \\
d  & a & b & c\\
c  & b & a & d \\
b  & d & c & a 
\end{pmatrix*}.
\]
Therefore, again by Proposition \ref{P:pairup}, $b \equiv -a \mod 5$ and as above, either  $P_{\rho}^{T} [A_{4}(b)](d) P_{\rho} =  A_{4}(c)$ or $P_{\rho}^{T} [A_{4}(b)](d) P_{\rho} = A_{4}(c)^{T}$, so $[A_{4}(b)](d)$ and $A_{4}(c)$ are spectrally equivalent.

\noindent \textbf{3 of a kind:}  This is impossible, since if $3$ entries on the diagonal are equal, then the 4th entry must also be equal to them.

\noindent \textbf{2 pair:}  We claim that there are no non-symmetric $5 \times 5$ Hadamard matrices with two pair on the diagonal.  Consider the following incomplete matrix:
\[
\begin{pmatrix*}[r]
a  & c & d & b \\
*  & b & a &  \\
  & a & b &  \\
b  &  &  & a 
\end{pmatrix*}.
\]
We pivot on the $*$ position:  if that position is a $c$, then the remaining entries are determined, and will result in a symmetric matrix.  If instead, it is a $d$, we obtain:
\[
\begin{pmatrix*}[r]
a  & c & d & b \\
d  & b & a & c \\
c & a & b &  d\\
b  & d & c & a 
\end{pmatrix*}.
\]
This matrix has the property that row 1 pairs up with both rows 3 and 4.  Thus, by Corollary \ref{C:pairup}, the resulting matrix is not the core of a Hadamard matrix.

We will prove in Theorem \ref{T:fnperm} that for a symmetric $5 \times 5$ Hadamard matrix, the diagonal of the core must have two pair.  Note that this together with the argument just given implies that if $A$ and $B$ have common traces, then either both are symmetric or both are nonsymmetric, justifying the considering of symmetric versus nonsymmetric cases separately.

\noindent \textbf{1 pair, 2 singletons:}  Since the diagonal entries of a matrix may be permuted by conjugation by a permutation matrix, we may assume without loss of generality that the pair is in the first two positions on the diagonal.  Consider the two incomplete matrices:
\[ 
A_{1}(*) = 
\begin{pmatrix*}[r]
a  & * &  &  \\
  & a &  &  \\
  &  & b & a  \\
  &  &  a & c 
\end{pmatrix*}
\qquad
A_{2}(*) = 
\begin{pmatrix*}[r]
a  & * &  &  \\
  & a &  &  \\
  &  & c & a \\
  &  &  a & b 
\end{pmatrix*}.
\]
Recall that the trace determines the diagonal entries, so we only need to consider diagonal entries in the three symbols $a,b,c$. 
We pivot on the $*$ position; in both cases, choosing $d$ in the pivot position yields a matrix that cannot be completed to a latin square.  There are two possible completions of $A_{1}$:
\[ 
A_{1}(b) = 
\begin{pmatrix*}[r]
a  & b & c & d \\
c  & a & d & b \\
d  & c & b & a  \\
b  & d & a & c 
\end{pmatrix*}
\qquad
A_{1}(c) = 
\begin{pmatrix*}[r]
a  & c &  d & b \\
b  & a &  c & d \\
c  & d &  b & a \\
d  & b &  a & c 
\end{pmatrix*}.
\]
These are transposes of each other, and thus are spectrally equivalent.  Likewise, $A_{2}(b)^{T} = A_{2}(c)$ and so are spectrally equivalent.  We therefore need to verify that $A_{1}$ and $A_{2}$ are spectrally equivalent. We have that $A_{2}(b)$ can be obtained from $A_{1}(b)$ by interchanging the third and fourth rows and columns, i.e. if $\sigma$ is the permutation $(34)$, then we have
\[ P_{\sigma}^{T} A_{2}(b) P_{\sigma} = A_{1}(b). \]

\noindent \textbf{4 singletons:}  We claim that there are no $5 \times 5$ Hadamard matrices with distinct diagonal entries.  Consider the following incomplete matrix.
\[ A_{0}(*)
\begin{pmatrix*}[r]
a  & * &  &  \\
  & b &  &  \\
  &  & c &  \\
  &  &   & d 
\end{pmatrix*}.
\]
Filling in subsequent entries which are uniquely determined yields:
\[ A_{0}(d) = 
\begin{pmatrix*}[r]
a  & d & b & c \\
c  & b & d & a \\
d & a & c & b \\
b & c & a & d 
\end{pmatrix*}.
\]
Observe that row 1 pairs up with both row 2 and row 3, so by Corollary \ref{C:pairup}, this cannot be the core of a $5 \times 5$ Hadamard matrix.  Analogously, $A_{0}(c)$ cannot be the core of a $5 \times 5$ Hadamard matrix.

\end{proof}

\section{$n \times n$ symmetric Hadamard matrices} \label{S:nxn}

We now consider larger Hadamard matrices, but in a much smaller context.  In particular, we consider only Hadamard matrices which are symmetric and permutations of the Fourier matrix.

\begin{lemma} \label{L:fnperm}
Suppose $\tau$ and $\sigma$ are permutations of $\mathbb{Z}_{n}$ and
\[ P_{\tau}^{T} \mathcal{F}_{n} P_{\sigma} = \mathcal{F}_{n}. \]
Then there exists $p,r \in \mathbb{Z}_{n}$ such that $\tau(j) \equiv pj \mod n$, $\sigma(k) \equiv rk \mod n$, and $pr \equiv 1 \mod n$.
\end{lemma}

\begin{proof}
We have that
\[ P_{\tau}^{T} \mathcal{F}_{n} P_{\sigma} [j,k] = \mathcal{F}_{n}[ \tau(j), \sigma(k) ] = \dfrac{1}{\sqrt{n}} e^{2 \pi i \tau(j) \sigma(k)/n}. \]
Thus, we must have for every $j,k \in \mathbb{Z}_{n}$,
\[ \tau(j) \sigma(k) \equiv j k \mod n. \]
For $j,k = 1$, we obtain $\tau(1) \sigma(1) \equiv 1 \mod n$, so both $\tau(1)$ and $\sigma(1)$ are units in $\mathbb{Z}_{n}$.  Therefore we have for every $k \in \mathbb{Z}_{n}$, $ \sigma(k) \equiv \tau(1)^{-1} k \mod n$, so $r = \tau(1)^{-1}$.  Likewise, for every $j \in \mathbb{Z}_{n}$, $\tau(j) \equiv \sigma(1)^{-1} j \mod n$, so $p = \sigma(1)^{-1}$.  Finally, $\tau(1) \sigma(1) \equiv 1 \mod n$ implies that $p r \equiv 1 \mod n$.
\end{proof}

A mapping $\tau$ of $\mathbb{Z}_{n}$ given by $\tau(j) = mj \mod n$ is a permutation if and only if $m \in \mathbb{Z}_{n}$ is a unit.  For such a permutation, we will denote the matrix $P_{\tau}$ by $P_{(\cdot m)}$.  An easy calculation demonstrates the following important commutation relation:

\begin{lemma} \label{L:commute}
For any $n \in \mathbb{N}$ and any unit $m \in \mathbb{Z}_{n}$,
\[ \mathcal{F}_{n} P_{(\cdot m)}^{T} = P_{(\cdot m)} \mathcal{F}_{n}. \]
\end{lemma}

\begin{theorem} \label{T:fnperm}
Suppose $M$ is an $n \times n$ Hadamard matrix in dephased form which is symmetric and permutationally equivalent to $\mathcal{F}_{n}$.  Then there exists a permutation $\tau$ of $\{1, \dots , n-1\}$ and a unit $m \in \mathbb{Z}_{n}$ such that
\[ M = P_{\tau}^{T} \mathcal{F}_{n} P_{(\cdot m)} P_{\tau}.  \]
The converse is also true.
\end{theorem}

\begin{proof}
By hypothesis, there exists permutations $\sigma, \rho$ of $\mathbb{Z}_{n}$ such that $ M = P_{\sigma}^{T} \mathcal{F}_{n} P_{\rho}$.  Since $M$ is symmetric, we have
\[ \left( P_{\sigma}^{T} \mathcal{F}_{n} P_{\rho} \right)^{T} = P_{\rho}^{T} \mathcal{F}_{n} P_{\sigma} = P_{\sigma}^{T} \mathcal{F}_{n} P_{\rho}, \]
from which it follows that $(P_{\rho} P_{\sigma}^{T}) \mathcal{F}_{n} (P_{\rho} P_{\sigma}^{T} ) = \mathcal{F}_{n}$.  Thus, by Lemma \ref{L:fnperm}, $P_{\rho} P_{\sigma}^{T} = P_{(\cdot m)}$ for some unit $m \in \mathbb{Z}_{n}$.  It follows that $P_{\rho} = P_{(\cdot m)} P_{\sigma}$ and  $M = P_{\sigma}^{T} \mathcal{F}_{n} P_{(\cdot m)} P_{\sigma}$.

Conversely,  if $\tau$ is a permutation of $\{1, \dots, n-1\}$ and $m \in \mathbb{Z}_{n}$ is a unit, then $M = P_{\tau}^{T} \mathcal{F}_{n} P_{(\cdot m)} P_{\tau}$ is a unitary matrix with all entries having complex modulus $\frac{1}{\sqrt{n}}$ and so is Hadamard.  In addition, it is symmetric since
\[ M^{T} = P_{\tau}^{T} P_{(\cdot m)}^{T} \mathcal{F}_{n} P_{\tau} = P_{\tau}^{T} \mathcal{F}_{n} P_{(\cdot m)} P_{\tau} \]
by Lemma \ref{L:commute}.  Finally, it is easy to check that it is in dephased form.
\end{proof}

\begin{corollary} \label{C:fnperm}
Suppose $M$ is an $n \times n$ Hadamard matrix in dephased form which is symmetric and permutationally equivalent to $\mathcal{F}_{n}$.   Then there exists a unit $m \in \mathbb{Z}_{n}$ such that $M$ is spectrally equivalent to $\mathcal{F}_{n} P_{(\cdot m)}$.
\end{corollary}

\begin{corollary}  \label{C:spectral}
Suppose $m \in \mathbb{Z}_{n}$ is a unit and $M = \mathcal{F}_{n} P_{(\cdot m)}$.  Then $M^{2} = P_{(\cdot (n-1))}$, $M^{4} = I$, and therefore, the spectrum of $M$ is contained in $\{ 1, -1, i, -i \}$.
\end{corollary}

\begin{proof}
We calculate:
\[ M^{2} = \mathcal{F}_{n} P_{(\cdot m)}  \mathcal{F}_{n} P_{(\cdot m)}  = \mathcal{F}_{n}^{2} P_{(\cdot m)}^{T} P_{(\cdot m)} \]
by Lemma \ref{L:commute}.  It is well known  (see e.g. \cite{McPa72}) that $\mathcal{F}_{n}^{2} = P_{(\cdot (n-1))}$.  Likewise, $M^{4} = \mathcal{F}_{n}^{4} = I$.  The result follows by the spectral mapping theorem.
\end{proof}

We now aim to calculate the multiplicities of the eigenvalues of $M = \mathcal{F}_{n} P_{(\cdot m)}$.  This will be accomplished in several steps, using known results on the eigenvalue multiplicities of $\mathcal{F}_{n}$ and Gauss sums.  We begin with Gauss sums, which is the trace of $M$.

\begin{definition}
Let $m,n \in \mathbb{Z}$ with $n>0$. Then the quadratic Gauss sum is
\[ g(m;n) = \sum_{j=0}^{n-1} e^{2 \pi i \frac{mj^2}{n} } . \]
\end{definition}

\begin{lemma} \label{L:trace}
The trace of the permutation of the Fourier matrix by $P_{(\cdot m)}$ is
\[ tr(\mathcal{F}_{n} P_{(\cdot m)} ) = \dfrac{1}{\sqrt{n}} g(m; n). \]
\end{lemma}

\begin{proof}
We calculate:
\begin{align*}
tr(\mathcal{F}_{n} P_{(\cdot m)}) &= \sum_{j=0}^{n-1} \mathcal{F}_{n} P_{(\cdot m)}[j,j] \\
&= \sum_{j=0}^{n-1} \mathcal{F}_{n} [j, mj] \\
&= \dfrac{1}{\sqrt{n}} \sum_{j=0}^{n-1} e^{2 \pi i j mj/n} \\
&= \dfrac{1}{\sqrt{n}} g(m; n).
\end{align*}
\end{proof}

The following Definition and Theorem are obtained from \cite{BEW98} and they determine the trace of $M$.  Using the value of the trace of $M$, we will be able to calculate the multiplicities.

\begin{definition}
Assume $m$ and $n$ are positive integers. Then define the Jacobi symbol by
\[ \left( \frac{n}{m} \right) = \left( \frac{n}{p_1} \right)^{\alpha_1} \left( \frac{n}{p_2} \right)^{\alpha_2} ... \left( \frac{n}{p_n} \right)^{\alpha_n}, \]
Where the prime factorization of $m$ is $m= p_1^{\alpha_1} p_2^{\alpha_2} ...p_n^{\alpha_n}$ and for a prime $p$, the Legendre symbol is
\[ \left( \frac{n}{p} \right), \]
which is $0$ if $n \equiv 0 \mod p$, $1$ if $n$ is a perfect square modulo $p$, and $-1$ otherwise.
\end{definition}

\begin{theorem} \label{book}
Let $m,n \in \mathbb{Z}$ be co-prime with $n>0$.

If $n \equiv 2 (\mod 4) $ then $g(m,n)=0$.

If $n \equiv 0 (\mod 4) $ then $g(m,n) = \left( \frac{n}{m} \right) (1+i^m) \sqrt{n}$.

If $n \equiv 1 (\mod 4) $ then $g(m,n) = \left( \frac{m}{n} \right)  \sqrt{n}$.

If $n \equiv 3 (\mod 4) $ then $g(m,n) = \left( \frac{m}{n} \right)  i\sqrt{n}$.

In the above formulas, $\left( \frac{a}{b} \right)$ denotes the Jacobi symbol.
\end{theorem}

We wish to compare the eigenvalues, and multiplicities, of the three matrices $\mathcal{F}_{n}$, $M$, and $P_{(\cdot (n-1))}$.   We know that $\sigma(\mathcal{F}_{n}) = \{ 1 , -1 , i , -i \}$, as well as the multiplicities from \cite{McPa72}, which we will specify below.  We also know that $\sigma(P_{(\cdot(n-1))}) = \{1, -1\}$.  For convenience,  let the multiplicities of eigenvalue $\lambda$ be denoted by $s_{\lambda}$ for $P_{(\cdot (n-1))}$, $t_{\lambda}$ for $M$, and $f_{\lambda}$ for $\mathcal{F}_{n}$.  The multiplicities of these three matrices are related as follows.

\begin{lemma} \label{L:mult}
The multiplicities  of the eigenvalues satisfy the following equations:
\[ t_{1} + t_{-1} = s_{1} = f_{1} + f_{-1}; \quad \text{ and } \quad t_{i} + t_{-i} = s_{-1} = f_{i} + f_{-i}. \]
\end{lemma}

\begin{proof}
If $v$ is an eigenvector for $M$ with eigenvalue $i$ or $-i$, then by Corollary \ref{C:spectral}, we have
\[ P_{(\cdot (n-1))} v = M^2 v = -1 v \]
so $v$ is an eigenvector for $P_{(\cdot (n-1))}$ with eigenvalue $-1$.  Similarly, if $v$ is a eigenvector for $M$ with eigenvalue $1$ or $-1$, then $v$ is an eigenvector for $P_{(\cdot (n-1))}$.  Therefore, we have that $t_{1} + t_{-1} \leq s_{1}$ and $t_{i} + t_{-i} \leq s_{-1}$.  However, all multiplicities must sum to $n$, so the inequalities must be equalities.

The same argument applies to the eigenvectors for $\mathcal{F}_{n}$.
\end{proof}

\begin{theorem} \label{T:nxnspectrum}
Let $n \in \mathbb{N}$ and $1 \neq m \in \mathbb{Z}_{n}$ be a unit.  Let $M = \mathcal{F}_{n} P_{(\cdot m)}$, and let $t_{\lambda}$ denote the multipliticity of the eigenvalue $\lambda$ for $M$.  The multiplicities are given by the following table:
\begin{center}
	\begin{tabular}{|| l | l || l | r || r | r | r | r ||}
	\hline
	$n$ & $m$ & $(\frac{n}{m}),(\frac{m}{n})$ & $tr(M)$ & $t_1$ & $t_{-1}$ & $t_i$ & $t_{-i}$ \\ \hline
	$4k$ & $4l +1$ & $(\frac{n}{m})=1$ & $1+i$ & k+1 & k & k & k-1  \\ \hline
	$4k$ & $4l + 1$ & $(\frac{n}{m})=-1$ & $-1-i $ & k & k+1 & k  & k-1 \\ \hline
	$4k$ & $4l + 3$ & $(\frac{n}{m})=1$ & $1-i $ & k+1 & k & k-1 & k  \\ \hline
	$4k$ & $4l + 3$ & $(\frac{n}{m})=-1$ & $-1+i $ & k & k+1 & k-1  & k \\ \hline
	$4k+1$ & $ $ & $(\frac{m}{n})=1 $ & $1 $ & k+1 & k & k & k \\ \hline
	$4k+1$ & $ $ & $(\frac{m}{n})=-1 $ & $-1 $ & k & k+1  & k & k \\ \hline
	$4k+2$ & $ $ & $ $ & $0 $ & k+1 & k+1 & k & k  \\ \hline		
	$4k+3$ & $ $ & $(\frac{m}{n})=1$ & $i $ & k+1 & k+1 & k+1 & k  \\ \hline		
	$4k+3$ & $ $ & $(\frac{m}{n})=-1$ & $-i $ & k+1 & k+1 & k & k+1  \\ \hline		
	\end{tabular}
\end{center} 

\end{theorem}

\begin{proof}
The proof proceeds on a case by case basis.  We present the details of the case $n = 4k$; all other cases are similar.  Note that since $m$ is a unit in $\mathbb{Z}_{n}$, $n$ and $m$ are co-prime, and hence both $(\frac{n}{m})$ and $(\frac{m}{n})$ are $\pm 1$.  We use the fact that we know the multiplicities for $\mathcal{F}_{n}$; as above, let $f_{\lambda}$ denote the multiplicity of the eigenvalue $\lambda$ for $\mathcal{F}_{n}$.  The multiplicities for $n = 4k$ are: $f_{1} = k+1$, $f_{-1} = k$, $f_{i} = k$, and $f_{-i} = k-1$, as demonstrated in \cite{McPa72}.

Combining Lemma \ref{L:trace} with Theorem 1.5.4 in \cite{BEW98}, we obtain
\begin{align*}
t_{1} - t_{-1} + i t_{i} - i t_{-i} &= tr(\mathcal{F}_{n} P_{(\cdot m)}) \\
&= \sqrt{n} g(m;n) \\
&= (\frac{n}{m})( 1 + i^{m}).
\end{align*}
 We equate the real and imaginary parts, then combine with Lemma \ref{L:mult} to obtain the following two systems of equations:
\begin{align*}
t_{1} + t_{-1} &= 2k + 1 & t_{i} + t_{-i} &= 2k -1 \\
t_{1} - t_{-1} &= (\frac{n}{m}) &  i t_{i} - i t_{-i} &= (\frac{n}{m}) i^{m}.
\end{align*}
Solving these systems of equations yields the result.
\end{proof}

\begin{corollary}
Suppose $M$ and $N$ are both $n \times n$ Hadamard matrices in dephased form, and are symmetric permutations of the Fourier matrix.  If $tr(M) = tr(N)$, then $M$ and $N$ are spectrally equivalent.
\end{corollary}

\begin{proof}
By Theorem \ref{T:fnperm}, there exists permutations $\tau_{1}$ and $\tau_{2}$, and units $m, m' \in \mathbb{Z}_{n}$ such that
\[ M = P_{\tau_{1}}^{T} \mathcal{F}_{n} P_{(\cdot m)} P_{\tau_{1}}^{T} \quad \text{ and } N = P_{\tau_{2}}^{T} \mathcal{F}_{n} P_{(\cdot m')} P_{\tau_{2}}^{T}. \]
Thus, $M$ is spectrally equivalent to $\mathcal{F}_{n} P_{(\cdot m)}$ and $N$ is spectrally equivalent to $\mathcal{F}_{n} P_{(\cdot m')}$.  Moreover, $tr(\mathcal{F}_{n} P_{(\cdot m)}) = tr(\mathcal{F}_{n} P_{(\cdot m')})$. Observing in the table in Theorem \ref{T:nxnspectrum} that the two traces being equal means that $m \equiv m' \mod 4$ and either $(\frac{n}{m}) = (\frac{n}{m'})$ or $(\frac{m}{n}) = (\frac{m'}{n})$, we obtain that the multiplicities of $\mathcal{F}_{n} P_{(\cdot m)}$ and $\mathcal{F}_{n} P_{(\cdot m')}$ are equal, and hence are spectrally equivalent.
\end{proof}

\section{Examples}
\subsection{Example 1} 
In Theorem \ref{T:nxnspectrum}, for the case $n=4k$ it appears that there are four possibilities for the multiplicities of the eigenvalues of a Hadamard matrix which is a symmetric permutation of the Fourier matrix. Consider the case $n=12$. Let $\F_{12}$ denote the $12 \times 12$ Fourier matrix, and $P_{(\cdot m)}$ denote the multiplicative permutation taking row $k$ to row $km$. Then the following table lists the multiplicities $m_x$ of the eigenvalues $x$ of $P_{(\cdot m)}^T \F_{12}$ for various $m$ mutually prime with $12$:
\begin{center}
	\begin{tabular}{| l | l | l | l | l |}
	\hline
	Matrix & $m_1$ & $m_{-1}$ & $m_i$ & $m_{-i}$ \\ \hline
	$ \F_{12}$ & 4 & 3 & 3  & 2 \\ \hline
	$P_{(\cdot 5)}^T \F_{12} $ & 3 & 4  & 2 & 3 \\ \hline
	$P_{(\cdot 7)}^T \F_{12}$ & 3 & 4 & 3 & 2  \\ \hline		
	$P_{(\cdot 11)}^T \F_{12}$ & 4 & 3 & 2  & 3  \\ \hline		
	\end{tabular}
\end{center} 
The above table was filled in using Theorem \ref{T:nxnspectrum}, and demonstrates an example of a Fourier matrix whose symmetric permutations represent the full set of possibilities outlined in the theorem.

\subsection{Example 2}
Using Theorem \ref{T:nxnspectrum}, we can find the spectrum of the symmetric $5 \times 5$ dephased Hadamard matrices. Using Theorem \ref{T:5x5H} and its proof, we can find the spectrum of the rest of the $5 \times 5$ dephased Hadamard matrices. To do this, we use a computer algebra system to find the spectrum of one such matrix from each equivalence class. As outlined in the proof of \ref{T:5x5H}, there are two possibilities: four of a kind on the diagonal, or a singleton and two pairs on the diagonal. We first treat the four of a kind case. Since the trace determines the equivalence class, we need one non-symmetric $5 \times 5$ dephased Hadamard matrix with each of the four possible four of a kind type diagonals. Let $\omega = \exp \left( \frac{2 \pi i}{5} \right) $. For example, observing the five by five Fourier matrix
\[
\F_5 = \frac{1}{\sqrt{5}}
\begin{pmatrix*}
1 & 1 & 1 & 1 & 1 \\
1 & \omega  & \omega^2 & \omega^3 & \omega^4  \\
1 & \omega^2  & \omega^4  & \omega & \omega^3  \\
1 &  \omega^3 & \omega & \omega^4 & \omega^2  \\
1 &  \omega^4 & \omega^3 & \omega^2 & \omega 
\end{pmatrix*},
\]
we notice that switching the third and fourth rows yields a non-symmetric $5 \times 5$ dephased Hadamard matrix with four of a kind type diagonal. Namely, if $P_{(*)}$ is the identity matrix permuted by the (column) permutation $(*)$, then
\[
P_{(3 4)}^T \F_5 = \frac{1}{\sqrt{5}}
\begin{pmatrix*}
1 & 1 & 1 & 1 & 1 \\
1 & \omega  & \omega^2 & \omega^3 & \omega^4  \\
1 &  \omega^3 & \omega & \omega^4 & \omega^2  \\
1 & \omega^2  & \omega^4  & \omega & \omega^3  \\
1 &  \omega^4 & \omega^3 & \omega^2 & \omega 
\end{pmatrix*}.
\]
We list one matrix from each such class and their respective spectra (the first is the one shown above). Let $q=\exp \left( \pi i /5 \right)$. The classes are represented by
\[
P_{(2 3)}^T \F_5 = \frac{1}{\sqrt{5}}
\begin{pmatrix*}
1 & 1 & 1 & 1 & 1 \\
1 & \omega  & \omega^2 & \omega^3 & \omega^4  \\
1 &  \omega^3 & \omega & \omega^4 & \omega^2  \\
1 & \omega^2  & \omega^4  & \omega & \omega^3  \\
1 &  \omega^4 & \omega^3 & \omega^2 & \omega 
\end{pmatrix*}
\]
with spectrum
\[
\left\lbrace 1,1,-1, -\frac25 +\frac45 q -\frac15 q^2 +\frac35 q^3 \pm \frac15 \sqrt{35-25q+20q^2 -20q^3 +25q^4}  \right\rbrace,
\]
\[
P_{(1 4)}^T \F_5 = \frac{1}{\sqrt{5}}
\begin{pmatrix*}
1 & 1 & 1 & 1 & 1 \\
1 &  \omega^4 & \omega^3 & \omega^2 & \omega \\
1 & \omega^2  & \omega^4  & \omega & \omega^3  \\
1 &  \omega^3 & \omega & \omega^4 & \omega^2  \\
1 & \omega  & \omega^2 & \omega^3 & \omega^4   
\end{pmatrix*}
\]
with spectrum
\[
\left\lbrace  1,1,-1, \frac25 -\frac45 q +\frac15 q^2 -\frac35 q^3 \pm \frac15 \sqrt{35-25q+20q^2 -20q^3 +25q^4} \right\rbrace,
\]
\[
P_{(1 2)(3 4)}^T \F_5 = \frac{1}{\sqrt{5}}
\begin{pmatrix*}
1 & 1 & 1 & 1 & 1 \\
1 & \omega^2  & \omega^4  & \omega & \omega^3  \\
1 & \omega  & \omega^2 & \omega^3 & \omega^4  \\
1 &  \omega^4 & \omega^3 & \omega^2 & \omega  \\
1 &  \omega^3 & \omega & \omega^4 & \omega^2  
\end{pmatrix*}
\]
with spectrum
\[
\left\lbrace 1,-1,-1, \frac15 q^3 -\frac25 q +\frac35 q^2 +\frac15 \pm \frac15 \sqrt{-5q+5 q^4 +20 -10q^3 +10q^2}   \right\rbrace,
\]
and
\[
P_{(1 3)(2 4)}^T \F_5 = \frac{1}{\sqrt{5}}
\begin{pmatrix*}
1 & 1 & 1 & 1 & 1 \\
1 &  \omega^3 & \omega & \omega^4 & \omega^2  \\
1 &  \omega^4 & \omega^3 & \omega^2 & \omega \\
1 & \omega  & \omega^2 & \omega^3 & \omega^4  \\
1 & \omega^2  & \omega^4  & \omega & \omega^3  
\end{pmatrix*}
\]
with spectrum
\[
\left\lbrace 1,-1,-1,-\frac15 q^3 +\frac25 q -\frac35 q^2 -\frac15 \pm \frac15 \sqrt{-5q+5 q^4 +20 -10q^3 +10q^2} \right\rbrace.
\]
Next, we consider one pair, two singletons case. Let $A$ be such a matrix. From the proof of Theorem \ref{T:5x5H}, we have
\[ 
\sqrt{5} A = 
\begin{pmatrix*}[r]
1 & 1 & 1 & 1 & 1 \\
1 & a  & b & c & d \\
1 & c  & a & d & b \\
1 & d  & c & b & a  \\
1 & b  & d & a & c 
\end{pmatrix*}
\]
or its transpose, where $a,b,c,d$ are fifth roots of unity. The orthogonality condition of unitary matrices (applied to rows two and three) yields
\[
0 = 1 + a\bar{c} + b\bar{a} + c\bar{d} + d\bar{b}.
\]
We shall show that a necessary condition for $A$ to be unitary is that $a=\bar{d}$. Assume otherwise. Then $a=\bar{c}$ or $a=\bar{b}$. In the first case, $d=\bar{b}$, so
\[
0=1+1+b\bar{a}+c\bar{d}+1,
\]
which is impossible by the triangle inequality. In the second case, $d=\bar{c}$, which is also impossible by the triangle inequality. Therefore, $a=\bar{d}$. This proves that any Hadamard matrix in the one pair, two singletons class must have the property that the element appearing twice on the diagonal is the conjugate of the element not appearing on the diagonal. Therefore, there are only four possible traces for such a matrix: once the element not appearing on the diagonal is fixed, the trace is completely determined. We write one matrix from each such equivalence class and their respective spectra. Let $q=\exp \left( \pi i /5 \right)$. The classes are represented by
\[
P_{(3 4)}^T \F_5 = \frac{1}{\sqrt{5}}
\begin{pmatrix*}
1 & 1 & 1 & 1 & 1 \\
1 & \omega  & \omega^2 & \omega^3 & \omega^4  \\
1 & \omega^2  & \omega^4  & \omega & \omega^3  \\
1 &  \omega^4 & \omega^3 & \omega^2 & \omega \\
1 &  \omega^3 & \omega & \omega^4 & \omega^2  
\end{pmatrix*}
\]
with spectrum
\[
\left\lbrace 1,-1,q^2,\frac{1}{10}q^3 -\frac15 q -\frac15 q^2 +\frac{1}{10} \pm \frac{1}{10} \sqrt{-5q+5-90q^3} \right\rbrace,
\]
\[
P_{(1 3)}^T \F_5 = \frac{1}{\sqrt{5}}
\begin{pmatrix*}
1 & 1 & 1 & 1 & 1 \\
1 &  \omega^3 & \omega & \omega^4 & \omega^2  \\
1 & \omega^2  & \omega^4  & \omega & \omega^3  \\
1 & \omega  & \omega^2 & \omega^3 & \omega^4  \\
1 &  \omega^4 & \omega^3 & \omega^2 & \omega 
\end{pmatrix*}
\]
with spectrum
\[
\left\lbrace 1,-1,-q^3,\frac25 q^3 +\frac15 q-\frac{3}{10}q^2 -\frac{1}{10} \pm \frac{1}{10} \sqrt{-20 q +25q^4 +25 -20q^3 +110q^2}  \right\rbrace,
\]
\[
P_{(1 2 3)}^T \F_5 = \frac{1}{\sqrt{5}}
\begin{pmatrix*}
1 & 1 & 1 & 1 & 1 \\
1 &  \omega^3 & \omega & \omega^4 & \omega^2  \\
1 & \omega  & \omega^2 & \omega^3 & \omega^4  \\
1 & \omega^2  & \omega^4  & \omega & \omega^3  \\
1 &  \omega^4 & \omega^3 & \omega^2 & \omega 
\end{pmatrix*}
\]
with spectrum
\[
\left\lbrace 1,-1,q, \frac{3}{10}q^3 -\frac15 -\frac{1}{10}q -\frac{1}{10}q^2 \pm \frac{1}{10} \sqrt{95q+90q^3 -5q^4 -90 -95q^2}  \right\rbrace,
\]
\[
P_{(1 2 4)}^T \F_5 = \frac{1}{\sqrt{5}}
\begin{pmatrix*}
1 & 1 & 1 & 1 & 1 \\
1 &  \omega^4 & \omega^3 & \omega^2 & \omega \\
1 & \omega  & \omega^2 & \omega^3 & \omega^4  \\
1 &  \omega^3 & \omega & \omega^4 & \omega^2  \\
1 & \omega^2  & \omega^4  & \omega & \omega^3  
\end{pmatrix*}
\]
with spectrum
\[
\left\lbrace 1,-1,-q^4, \frac15 q^3 -\frac{3}{10} +\frac{1}{10} q -\frac{2}{5}q^2 \pm \frac{1}{10} \sqrt{-110q -20q^3 +20q^4 +25 +25q^2}  \right\rbrace.
\]
We present the spectrum of every dephased $5 \times 5$ Hadamard matrix in the table below: the symmetric ones are listed first, then the four of a kind asymmetric ones, then the pair, two singletons asymmetric ones. Let $\omega=\exp(2 \pi i /5)$. We organize the matrices by trace, but omit the factor of $\frac{1}{\sqrt5}$.
\begin{center}
	\begin{tabular}{| l | l |}
	\hline
	Trace ($\times \sqrt{5})$ & Spectrum  \\ \hline
	$ 1+ 2\omega +2\omega^4 $ & $\{ 1,1,-1,i,-i \}$\\ \hline
	$ 1+ 2\omega^2 +2\omega^3 $ &  $\{ 1,-1,-1,i,-i \}$ \\ \hline
	$ 1+ 4\omega  $ & $ \left\lbrace 1,1,-1, -\frac25 +\frac45 q -\frac15 q^2 +\frac35 q^3 \pm \frac15 \sqrt{35-25q+20q^2 -20q^3 +25q^4}  \right\rbrace $\\ \hline		
	$ 1+ 4\omega^4 $ &  $ \left\lbrace  1,1,-1, \frac25 -\frac45 q +\frac15 q^2 -\frac35 q^3 \pm \frac15 \sqrt{35-25q+20q^2 -20q^3 +25q^4} \right\rbrace $ \\ \hline	
	$ 1+ 4\omega^2 $ &  $ \left\lbrace 1,-1,-1, \frac15 q^3 -\frac25 q +\frac35 q^2 +\frac15 \pm \frac15 \sqrt{-5q+5 q^4 +20 -10q^3 +10q^2}   \right\rbrace $ \\ \hline
	$ 1+ 4\omega^3 $ &  $ \left\lbrace 1,-1,-1,-\frac15 q^3 +\frac25 q -\frac35 q^2 -\frac15 \pm \frac15 \sqrt{-5q+5 q^4 +20 -10q^3 +10q^2} \right\rbrace $ \\ \hline
	$ 1+ \omega + 2\omega^2 + \omega^4 $ &  $ \left\lbrace 1,-1,q^2,\frac{1}{10}q^3 -\frac15 q -\frac15 q^2 +\frac{1}{10} \pm \frac{1}{10} \sqrt{-5q+5-90q^3} \right\rbrace $ \\ \hline
	$ 1 + \omega + 2\omega^3  +\omega^4 $ &  $ \left\lbrace 1,-1,-q^3,\frac25 q^3 +\frac15 q-\frac{3}{10}q^2 -\frac{1}{10} \pm \frac{1}{10} \sqrt{-20 q +25q^4 +25 -20q^3 +110q^2}  \right\rbrace $ \\ \hline
	$ 1+ 2\omega + \omega^2  + \omega^3  $ &  $ \left\lbrace 1,-1,q, \frac{3}{10}q^3 -\frac15 -\frac{1}{10}q -\frac{1}{10}q^2 \pm \frac{1}{10} \sqrt{95q+90q^3 -5q^4 -90 -95q^2}  \right\rbrace $ \\ \hline
	$ 1+ \omega^2 + \omega^3 + 2 \omega^4 $ &  $ \left\lbrace 1,-1,-q^4, \frac15 q^3 -\frac{3}{10} +\frac{1}{10} q -\frac{2}{5}q^2 \pm \frac{1}{10} \sqrt{-110q -20q^3 +20q^4 +25 +25q^2}  \right\rbrace $ \\ \hline	
	\end{tabular}
\end{center}

The first two entries of the table correspond to matrices with two pair on the diagonal, and are hence symmetric.  The remaining entries in the table correspond to nonsymmetric matrices.  Notice that there are only two traces listed for symmetric $5\times5$ dephased Hadamard matrices. This follows from Theorem \ref{book} on Gauss sums: $tr(\F_5) = tr(\F_5 P_{(\cdot 4)} )$ and $tr(\F_5 P_{(\cdot 2)} ) = tr(\F_5 P_{(\cdot 3)} )$.

\begin{acknowledgements}
This work was partially supported by a grant from the Simons Foundation (\#228539 to Dorin Dutkay).
\end{acknowledgements}




\bibliographystyle{amsplain}
\bibliography{eframes}
\end{document}